\PassOptionsToPackage{dvipsnames}{xcolor}
\documentclass[10pt]{amsart}
\usepackage[margin=3cm]{geometry}
\usepackage{tikz}
\usepackage{amssymb}
\usetikzlibrary{arrows.meta}
\usetikzlibrary{positioning}
\usetikzlibrary{decorations.markings}
\usetikzlibrary{shapes.misc}
\usepackage{xcolor}
\definecolor{f1c1}{RGB}{173, 16, 90}
\definecolor{f1c2}{RGB}{11, 9, 86}
\definecolor{f1c3}{RGB}{63, 136, 129}
\definecolor{f4c1}{RGB}{172, 132, 42}
\definecolor{f4c2}{RGB}{166, 44, 206}
\definecolor{f5c1}{RGB}{235, 117, 163}
\definecolor{f7c1}{RGB}{234, 170, 154}
\definecolor{f7c2}{RGB}{172, 167, 196}
\definecolor{f7c3}{RGB}{179, 195, 207}
\usetikzlibrary{hobby}
\usetikzlibrary{patterns,patterns.meta}
\usepackage{float}

\usepackage[utf8]{inputenc} 
\usepackage{graphicx}
\usepackage{pgfplots}
\pgfplotsset{compat=1.18}
\usepackage{amsmath, amssymb, mathtools}
\usepackage{multicol}
\usepackage{hyperref}
\usepackage{cleveref}
\usepackage{dsfont}
\usepackage{xcolor}
\usepackage{tikz}
\usetikzlibrary{intersections,positioning}

\newtheorem{theorem}{Theorem}[section]

\newtheorem{lemma}{Lemma}[section]

\newtheorem*{maintheorem*}{Main Theorem}
\allowdisplaybreaks
\numberwithin{equation}{section}

\usepackage{todonotes}

\newcommand{\R}{\mathbb{R}}

\newcommand{\eps}{\varepsilon}

\newcommand{\dd}{\,\mathrm{d}}

\begin{document}

	\title[Overdetermined free boundary problems]{Symmetry results for some overdetermined obstacle problems} 
	
	\subjclass[2020]{35J86, 35N25, 35R35, 35B06, 35B35, 49J40.}
	\keywords{Obstacle problem, Serrin-type problem, two-phase problem, overdetermined problem.}

	\author[N. De Nitti]{Nicola De Nitti}
	\address[N. De Nitti]{Friedrich-Alexander-Universit\"{a}t Erlangen-N\"{u}rnberg, Department of Mathematics, Chair for Dynamics, Control, Machine Learning and Numerics (Alexander von Humboldt Professorship), Cauerstr. 11, 91058 Erlangen, Germany.}
	\email{nicola.de.nitti@fau.de}

	\author[S. Sakaguchi]{Shigeru Sakaguchi}
	\address[S. Sakaguchi]{Graduate School of Information Sciences, Tohoku University, Sendai 980-8579, Japan.}
	\email{sigersak@tohoku.ac.jp}
	
	\begin{abstract}
We establish symmetry results for two categories of overdetermined obstacle problems: a Serrin-type problem and a two-phase problem under the overdetermination that the interface serves as a level surface of the solution. The first proof avoids the method of moving planes; instead, it leverages the comparison principle and the fact that the solution of the obstacle problem is superharmonic within its domain. The second proof utilizes a suitable version of Serrin's method of moving planes. 
	\end{abstract}

	\maketitle

	\section{Overdetermined obstacle problems}
	\label{sec:intro}
The main aim of this paper is to study the effect of overdetermination on two types of obstacle problems: first, a Serrin-type obstacle problem and, secondly, a two-phase obstacle problem under the assumption that the interface serves as a level surface of the solution. 

\subsection{Serrin-type obstacle problem} 
\label{ssec:intro-serrin}

Let \(\Omega\) be a bounded  domain  in \(\mathbb{R}^{N}\), with $N \ge 2$, such that \(0 \in \Omega\) and consider 
\begin{align}\label{eq:op}
\begin{cases}
\min \{-\Delta u, u-\psi\}=0 & \text { in } \Omega, \\
u=0 & \text { on } \partial \Omega,
\end{cases}
\end{align}
which can be equivalently rewritten as 
\begin{align*}
\begin{cases} 
-\Delta u \ge 0 & \text{ in } \Omega, \\
u \ge \psi & \text{ in } \Omega, \\
(u-\psi) \, \Delta u = 0 & \text{ in } \Omega, \\
u= 0 & \text{ on } \partial \Omega,
\end{cases}
\end{align*}
subject to the additional Neumann boundary  condition
\begin{equation} \label{eq:neumann}
\partial_\nu u = c  \quad  \text { on } \partial \Omega,
\end{equation}
where   $\nu$ denotes the outward-pointing unit normal vector to $\partial \Omega$, $c$ is a real constant,  and the obstacle function \(\psi \in C^{2}(\bar{\Omega})\) satisfies the following hypotheses: 
\begin{align}\label{ass:o1}
&\overline{\{x\in\Omega : \psi(x) > 0\}} \subset \Omega;\\
\label{ass:o2}&\max \psi>0 \ \text { and } \ \psi \text { is radially symmetric with respect to the origin}.
\end{align}
The condition \eqref{ass:o1} ensures the existence of solutions of \eqref{eq:op}; assumption \eqref{ass:o2}, on the other hand, plays a key role in the first symmetry result. 

\begin{theorem}[Serrin-type overdetermined obstacle problem]\label{th:main1}
Let us assume that \(\partial\Omega\) is of class $C^1$, \(0 \in \Omega\) and $\psi$ satisfies \eqref{ass:o1} and \eqref{ass:o2}. If $u \in C^{1,1}(\Omega)\cap C^1(\bar\Omega)$ is the solution of \eqref{eq:op} subject to \eqref{eq:neumann}, then $u$ is radially symmetric with respect to the origin and \(\Omega\) is a ball centered at the origin.
\end{theorem}

The solution of the obstacle problem \eqref{eq:op} can be seen as the \emph{smallest superharmonic function lying above the obstacle}: more precisely, it is the unique minimizer of the Dirichlet energy 
\begin{align*}
\mathcal E[w] := \int_\Omega |\nabla w|^2 \dd x
\end{align*}
among all functions $w$ belonging to the set 
$$\mathcal K[\psi] := \{w \in H^1_0(\Omega): w \ge \psi \text{ a.e. in } \Omega\},$$
which is closed, convex,
and non-empty. The minimizer $u \in \mathcal K[\psi]$ of the Dirichlet energy $\mathcal E$ can be also characterized via the following \emph{elliptic variational inequality}: 
\begin{align*}
\int_\Omega \nabla u \cdot \nabla (v-u) \dd x \ge 0 \quad \mbox{ for every } v \in \mathcal K[\psi].
\end{align*}
The associated \emph{Euler-Lagrange equations},
\begin{align*}
\begin{cases}
-\Delta u \ge 0 & \text{ in } \Omega, \\
-\Delta u = 0 & \text{ in } \{ u > \psi\}, \\
u\ge \psi & \text{ in } \Omega, \\
u= 0 & \text{ on } \partial \Omega,
\end{cases}\end{align*}
indicate that the domain $\Omega$  is split into two regions: $\mathcal \{u > \psi\}$, in which $u$
is harmonic, and the \emph{coincidence set} $\{u = \psi\}$, in which the solution coincides with the obstacle. The interface
$\partial \{u > \psi\}$ that separates these two regions is called the \emph{free boundary}.
Using the properties resulting from the Euler-Lagrange equations, we can, in turn, deduce the \emph{complementarity problem} formulated in \eqref{eq:op}.

We refer to the classical texts  \cite{MR0679313, MR1786735,MR2962060,MR880369,MR1094820}  as well as to the more recent book \cite[Chapter 5]{MR4560756} and surveys \cite{JEDP_2018_A2_0, MR3855748} for further information on the obstacle problem (its equivalent formulations, the regularity of solutions, the regularity of the free boundary, etc.). Let us only conclude by mentioning that indeed there exists a unique solution $u \in C^{1,1}( \Omega)\cap C^1(\bar\Omega)$ of \eqref{eq:op} provided $\partial\Omega$ is of class $C^{1,\alpha}$ for some $0 < \alpha < 1$. For instance, see in particular \cite[Theorem 6.1, p. 125]{MR1786735} or \cite[Theorem 2.14, p. 42]{MR2962060} for the $C^{1,1}$ regularity.

The symmetry result in Theorem \ref{th:main1} is in the same spirit as Serrin's symmetry result for the torsional rigidity \cite{MR333220}: namely, that the overdetermined boundary value problem: if $u$ is a solution of 
\begin{align*}
\begin{cases}
-\Delta u =1 & \text{ in }  \Omega, \\
u=0 & \text{ on } \partial \Omega, \\
\partial_\nu u =c & \text { on } \partial \Omega,
\end{cases}
\end{align*}
with $c \in \R$, then $\Omega$ is a ball of radius \(R>0\) and, up to translations, \(u(x)=\frac{R^2-|x|^2}{2N}\). 

The main tool of Serrin's original proof is Alexandrov's moving plane technique (see \cite{MR143162,MR0150710}); several alternative strategies have been developed over the years. We refer to \cite{MR3802818} for a survey.

On the other hand, in this paper, we prove Theorem \ref{th:main1} by a direct approach, relying only on the comparison principle and on the fact that the solution of \eqref{eq:op} is superharmonic in $\Omega$ and lies above the obstacle $\psi$. Such an approach for overdetermined boundary value problems associated with the Poisson equation was already used by Onodera in \cite{MR3312968}. As a  by-product of our proof strategy, we deduce the following stability result (which, in fact, implies Theorem \ref{th:main1}). 

\begin{theorem}[Stability for the Serrin-type overdetermined obstacle problem]\label{th:main1-2}
Let us assume that \(\partial\Omega\) is of class $C^1$, \(0 \in \Omega\), and $\psi$ satisfies \eqref{ass:o1} and \eqref{ass:o2}. Among balls centered at the origin, let $B_\rho(0), B_R(0)$ be the largest ball contained in $\Omega$ and the smallest ball containing $\Omega$, respectively. If $u \in C^{1,1}(\Omega)\cap C^1(\bar\Omega)$, the solution of \eqref{eq:op}, satisfies
\begin{equation}
\label{close to a constant}
|\partial_\nu u -c| \le \varepsilon\ \mbox{ on } \partial\Omega
\end{equation}
for some constants $c < 0 < \varepsilon$, then 
\begin{equation}
\label{stability}
R-\rho \le K\varepsilon,
\end{equation}
where $K > 0$ is a constant depending only on $N$, the diameter of $\Omega$, and the obstacle function $\psi$ over $\{ x \in \Omega : \psi(x) > 0\}$.
\end{theorem} 

Stability results related to Serrin's overdetermined problem (without obstacles) have been obtained by \cite{MR1729395, MR2436453}.

\subsection{Two-phase obstacle problems}
\label{ssec:two-phase}

Secondly, we study two-phase obstacle problems. Let $D$ be a bounded domain in  \(\mathbb{R}^{N}\), with $N \ge 2$, such that \(0 \in D\), and denote by $\sigma=\sigma(x)$, with $ x\in\mathbb R^N$, the conductivity distribution of the whole medium given by
\begin{equation}
\sigma(x) =\begin{cases}
\sigma_+& \text{ if } x \in D,\\
\sigma_-& \text{ if } x \in \mathbb R^N\setminus D,
\end{cases}\label{conductivity}
\end{equation}
where $\sigma_+, \sigma_-$ are positive constants with $\sigma_+\not=\sigma_-$. Given the obstacle function   \(\psi \in C^{2}(\bar{D})\cap C(\mathbb R^N)\) satisfying \eqref{ass:o2} and
\begin{equation}\label{supported inside D}
\mathrm{supp}(\max\{\psi, 0\}) :=\overline{\{ x \in \mathbb R^N : \psi(x) > 0\}} \subset D, 
\end{equation}
we consider the following problems: 
\begin{align}\label{eq:op2}
\begin{cases}
\min \{-\mbox{div}(\sigma\nabla u), u-\psi\}=0 & \text { in } B_L(0), \\
u=0 & \text { on } \partial B_L(0),
\end{cases}
\end{align}
where $L > 0$ is chosen to have $\overline{D}\subset B_L(0)$, and
\begin{align}\label{eq:op3}
\begin{cases}
\min \{-\mbox{div}(\sigma\nabla u), u-\psi\}=0 & \text { in } \mathbb R^N, \\
u \to 0 & \mbox{ as } |x| \to \infty.
\end{cases}
\end{align}
Here, $B_r(z)$ denotes an open ball in $\mathbb R^N$ with radius $r >0$ centered at a point $z$. The assumption \eqref{supported inside D} implies that the coincidence sets are contained in $D$.
The overdetermination in these problems is given by imposing the additional Dirichlet boundary  condition
\begin{equation} \label{eq:dirichlet}
 u = d \qquad  \text { on } \partial D, 
\end{equation}
for some constant $d > 0$.  We remark that  $0 < d < \max \psi$, since $u < \max \psi$ outside the largest ball centered at the origin and containing the support of $\max\{\psi, 0\}$ (see \eqref{shape of the summit of the solution} for a more precise estimate). Concerning these problems, we prove the following symmetry results.

\begin{theorem}[Overdetermined two-phase obstacle problem]\label{th:main2}
Let us assume that \(\partial D\) is of class $C^{2,\alpha}$ for some $0 < \alpha < 1$, \(0 \in D\), and $\psi$ satisfies \eqref{ass:o2} and \eqref{supported inside D}. If $u \in C^{1,1}(\overline{D})\cap C^{0,1}(\overline{B_L(0)})$ is the solution of \eqref{eq:op2} subject to \eqref{eq:dirichlet}, then $u$ is radially symmetric with respect to the origin and \(D\) is a ball centered at the origin.
\end{theorem}

\begin{theorem}[Overdetermined two-phase obstacle problem in $\R^N$]\label{th:main3} Let us assume that $N \ge 3$, \(\partial D\) is of class $C^{2,\alpha}$ for some $0 < \alpha < 1$, \(0 \in D\), and $\psi$ satisfies \eqref{ass:o2} and \eqref{supported inside D}. If $u \in C^{1,1}(\overline{D})\cap C^{0,1}(\mathbb R^N)$ is the solution of \eqref{eq:op3} subject to \eqref{eq:dirichlet}, then $u$ is radially symmetric with respect to the origin and \(D\) is a ball centered at the origin.
\end{theorem}

The case where $N=2$ in Theorem \ref{th:main3} is excluded since problem \eqref{eq:op3} has no solutions in two dimensions. Theorems \ref{th:main2}  and \ref{th:main3} are proven by
the method of moving planes due to Serrin with the aid of approximate solutions via penalized equations.  Such an approach for overdetermined two-phase boundary value problems (without obstacles)  was already introduced by Kang and the second author in \cite{MR4517773}.

\subsection*{Outline}

The following three sections are devoted to the proofs of Theorems \ref{th:main1}, \ref{th:main1-2}, \ref{th:main2}, and \ref{th:main3}.

\section{Proof of Theorems \ref{th:main1} and \ref{th:main1-2}}
\label{sec:proof1}

\begin{proof}[Proof of Theorem \ref{th:main1}]

\noindent
\textbf{Step 1.} \emph{Auxiliary problems in symmetric domains.} Let $u$ be the solution of \eqref{eq:op} subject to \eqref{eq:neumann}.
Let us fix \(0<\rho  \leq R<\infty\) satisfying 
\begin{align*}
B_\rho(0) \subset \Omega &\quad \text{ and }\quad \partial B_\rho(0) \cap \partial \Omega \neq \emptyset, \\
B_{R}(0) \supset \Omega  &\quad \text{ and }\quad \partial B_R(0) \cap \partial \Omega \neq \emptyset;
\end{align*}
since $\psi$ is assumed to be radially symmetric, we can also suppose that the support of its positive part $\max\{\psi, 0\}$ is contained in $B_\rho(0)$.
Let \(u_{\rho}\) and \(u_{R}\) be the solutions of the obstacle problems \eqref{eq:op} where $\Omega$ is replaced by the domains \(B_{\rho}(0)\) and \(B_{R}(0)\), respectively. By the comparison principle, we have 
\begin{align*}
\begin{cases}
u_{\rho} \leq u \leq u_{R} & \text { in } B_{\rho}(0), \\
u \leq u_{R} & \text { in } \Omega .
\end{cases}
\end{align*}

\begin{figure}[h]
	\centering
\begin{tikzpicture}[ultra thick, font = \Large, scale=0.6]
\draw [pattern=crosshatch, pattern color=blue] (0,0) circle (5cm);
\path[xshift = -5cm, yshift = -5cm, draw = black, fill = pink, fill opacity=1, use Hobby shortcut, closed=true]
(7.381, 4.616) .. (8.504, 6.233) .. (8.449, 7.767) .. (7.025, 8.671) .. (4.888, 9.328) .. (2.889, 8.917) .. (1.519, 7.657) .. (0.807, 5.493) .. (0.533, 3.932) .. (0.697, 2.808) .. (2.423, 1.658) .. (3.71, 1.055) .. (5.08, 0.973) .. (6.231, 1.658) .. (6.888, 2.945);
\path[xshift = -5cm, yshift = -5cm, draw = black,  use Hobby shortcut, closed=true, fill=white, fill opacity=1 ]
(2.204, 5.959) .. (1.738, 5.329) .. (1.711, 5.0) .. (1.848, 4.671) .. (1.985, 4.123) .. (2.642, 3.822) .. (2.943, 4.562) .. (2.532, 5.055) .. (2.286, 5.411) .. (2.286, 5.931);
\path[xshift = -5cm, yshift = -5cm, draw = black,  use Hobby shortcut, closed=true, fill=white, pattern=crosshatch, pattern color=blue, fill opacity=1 ]
(2.204, 5.959) .. (1.738, 5.329) .. (1.711, 5.0) .. (1.848, 4.671) .. (1.985, 4.123) .. (2.642, 3.822) .. (2.943, 4.562) .. (2.532, 5.055) .. (2.286, 5.411) .. (2.286, 5.931);
\draw [fill=white]  (0,0) circle (2cm);
\draw [pattern=north west lines, pattern color=green!50!black]  (0,0) circle (2cm);
\draw [-latex] (0,0) node [below left] {0} -- node [above] {$\rho$} (2,0);
\draw [-latex] (0,0) -- (70:5) node [below = .8cm] {$R$};
\node at (-2,-2) {$\Omega$};
\end{tikzpicture}
\caption{Example of $\Omega$, $B_\rho(0)$ and $B_{R}(0)$.}
\label{fig:sets}
\end{figure}
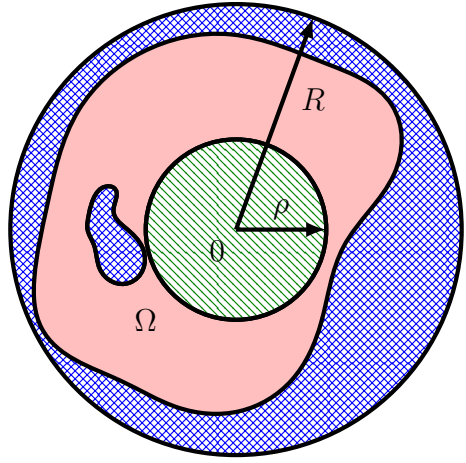

For the coincidence sets 
\begin{align*}
\mathcal I &=\{x \in \Omega: u(x)=\psi(x)\}, \\
\mathcal I_{\rho} &=\left\{x \in B_{\rho}(0): u_{\rho}(x)=\psi(x)\right\}, \\
\mathcal I_{R} &=\left\{x \in B_{R}(0): u_{R}(x)=\psi(x)\right\},
\end{align*}
the following inclusions hold: 
$$\mathcal I_{\rho} \supset  \mathcal I \supset \mathcal  I_{R}.$$

Let $\overline{B_{a}(0)}$ and $\overline{B_{b}(0)}$ be the smallest balls containing \(\mathcal I_{\rho}\) and \(\mathcal I_{R}\), respectively.
Then, since $B_R(0)\supset B_\rho(0)\supset\mathcal I_{\rho} \supset \mathcal  I_{R},$
\begin{equation}\label{order of key numbers}
0 < b \le a< \rho\le R.
\end{equation}
Moreover, since  \(u_{\rho}, u_{R}\) are radially symmetric and superharmonic in  \(B_{\rho}(0)\) and \(B_{R}(0)\), respectively, we notice that 
\begin{equation}
\label{key monotonicity}
 r^{N-1} \partial_r u_\rho, r^{N-1} \partial_r u_R \mbox{ are monotone nonincreasing in }  [0,\rho] \mbox{ and in } [0, R], \mbox{respectively}.
\end{equation}
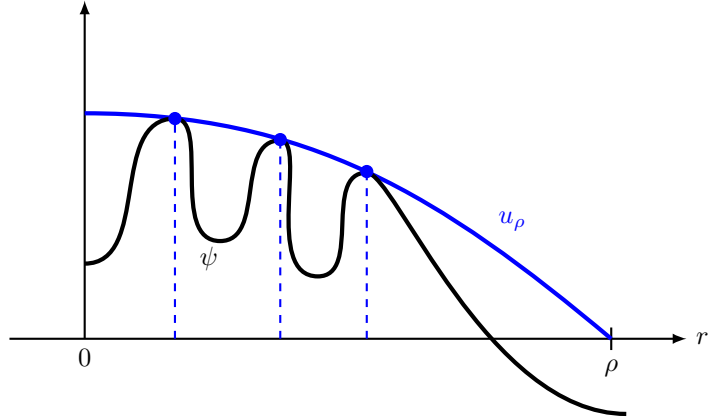
\begin{figure}[ht]
	\centering
	\begin{tikzpicture}[thick]
	
	\draw [-latex] (-1,0) -- (8,0) node [right] {$r$};
	\draw [-latex] (0,0) node [below] {$0$} -- (0,4.5);
	\draw (7,.15) -- +(0,-.3) node [below] {$\rho$};
	\draw [ultra thick,blue] (0,3) to [out = 0, in = 140] node [pos = .8, above = 0.25 cm] {$u_\rho$} node [black, pos = .2, below = 1.5cm] {$\psi$} (7,0);
	\draw [ultra thick,black] (0,1) to [out = 0, in = 180] (1.2,2.93) to [out = 0, in = 180, out looseness = .7] (1.8,1.3) to [out = 0, in = 180] (2.6,2.64) to [out = 0, in = 180, out looseness = .5] (3.1,.83) to [out = 0, in = 180] (3.75,2.21) to [out = 0, in = 180, out looseness = .2] (7.2,-1);
	\draw [blue, dashed] (1.2,0) node [below] {\,} -- +(0,2.93);
	\draw [blue, dashed] (2.6,0) node [below] {\,} -- +(0,2.65);
	\draw [blue, dashed] (3.75,0) node [below] {\,} -- +(0,2.225);
	\fill [blue] (1.2,2.93) circle (2.5pt);
	\fill [blue] (2.6,2.65) circle (2.5pt);
	\fill [blue] (3.75,2.225) circle (2.5pt);
	\end{tikzpicture}
	\caption{Monotonicity of $u_\rho$}
	\label{fig:psi-c}
\end{figure}
Hence, in particular
\begin{equation}
\label{strict inequalities without Hopf's boundary point lemma}
\partial_r u_\rho(\rho) < 0  \quad \mbox{ and } \quad \partial_r u_R(R) < 0,
\end{equation}
since  $u_R(0)= u_\rho(0)>0 $ and $u_R(R)= u_\rho(\rho)=0$.
Our aim is to show that, in fact, we can choose $\rho = R$ and thus $\Omega$ is a ball.

\textbf{Step 2.} \emph{Monotonicity-type result.} Since outside of the coincidence set the solution of the obstacle problem is harmonic, namely 
\begin{align*}
-\Delta u_{\rho}&=0 \quad \text { in } B_{\rho}(0) \setminus \overline{B_{a}(0)},\\
-\Delta u_{R}&=0  \quad \text { in } B_{R}(0) \setminus B_{b}(0), 
\end{align*}
we have that 
$u=u_{\rho}(r) \text { or } u_{R}(r)$  satisfies 
\begin{align*}
\partial_r\left( r^{N-1}\partial_r u\right) = 0 \quad  \text{(with } r=|x|).
\end{align*}
Thus, we compute  
\begin{align}
0&=\int_{a}^{\rho}\partial_r\left(r^{N-1} \partial_r u_{\rho}\right) \dd r=\rho^{N-1} \partial_r u_{\rho}(\rho)-a^{N-1} \partial_r u_{\rho}(a), \label{eq for rho}\\
0&=\int_{b}^{R}\partial_r\left(r^{N-1} \partial_r u_{R}\right) \dd r=R^{N-1} \partial_r u_{R}(R)-b^{N-1} \partial_r u_{R}(b).\label{eq for R}
\end{align}
We claim the following inequalities hold: 
 \begin{equation}\label{monotonicity of radial solutions}
  \partial_r u_{\rho}(\rho) \leq \partial_r u_{R}(R) < 0.
  \end{equation}
Indeed, by combining \eqref{strict inequalities without Hopf's boundary point lemma}, \eqref{eq for rho}, \eqref{eq for R}, \eqref{key monotonicity} with \eqref{order of key numbers}, we compute 
\begin{align}
-\partial_r u_{\rho}(\rho)&=-\left(\frac{a}{\rho}\right)^{N-1} \partial_r u_{\rho}(a) \ge -\left(\frac{b}{\rho}\right)^{N-1} \partial_r u_{\rho}(b)\label{chain of inequalities1} \\
&=-\left(\frac{b}{\rho}\right)^{N-1} \partial_r \psi(b) = -\left(\frac{b}{\rho}\right)^{N-1} \partial_r u_R(b)=-\left(\frac{R}{\rho}\right)^{N-1} \partial_r u_R(R)\label{chain of inequalities2}\\ 
&\ge -\partial_r u_{R} (R)\quad  ( > 0). \label{chain of inequalities3}
\end{align}

\textbf{Step 3.} \emph{Conclusion of the argument.} Since \(u \leq u_{R}\) in \(\Omega\) and \(u=0=u_{R}\) on \(\partial \Omega \cap \partial B_{R}(0)\neq \emptyset\), 
we have
\begin{equation} 
\label{estimate by the large ball}
-\partial_r u_{R}(R)=-\partial_\nu u_{R} \geqslant-\partial_\nu u=-c \text{ on } \partial \Omega \cap  \partial B_{R}(0).
\end{equation} 
Similarly, since \(u_{\rho} \leq u\) in \(B_{\rho}(0)\) and
\(u_{\rho}=0=u\) on \(\partial \Omega \cap \partial B_{\rho}(0)\neq \emptyset\),
we have
\begin{equation}
\label{estimate by the small ball}
-\partial_r u_{\rho} (\rho)=-\partial_\nu u_{\rho} \leq-\partial_\nu u=-c \text { on } \partial \Omega \cap \partial B_{\rho}(0).
\end{equation}
From this, we conclude \(-\partial_r u_{\rho}(\rho) \leq -c \leq-\partial_r u_{\rho}(R)\). Furthermore,  it follows from \eqref{monotonicity of radial solutions} that
\begin{align*}
-\partial_r u_{\rho}(\rho)=-c=- \partial_r u_{R}(R),
\end{align*}
which implies the equalities in \eqref{chain of inequalities1}--\eqref{chain of inequalities3} and, hence,
\begin{align*}
\rho=R.
\end{align*}
Thus \(\Omega=B_{\rho}(0)=B_{R}(0)\). 
\end{proof}

\begin{proof}[Proof of Theorem \ref{th:main1-2}] The computations \eqref{chain of inequalities1}--\eqref{chain of inequalities3} yield  
$$
-\partial_r u_{s} (s) > -\partial_r u_{t} (t) > 0 \quad \text{ if } 0 < s < t.
$$
By virtue of arguments in \eqref{chain of inequalities1}--\eqref{estimate by the small ball} with the assumption \eqref{close to a constant}, we deduce
$$
-c+\varepsilon\ge -\partial_r u_{\rho}(\rho)\ge -\left(\frac{R}{\rho}\right)^{N-1} \partial_r u_R(R) \ge -\partial_r u_{R} (R) > -c-\varepsilon
$$
and, hence,
$$
2\varepsilon \ge \frac{(R^{N-1}-\rho^{N-1})(- \partial_r u_R(R) )}{\rho^{N-1}}.
$$
Thus, 
$$
R-\rho\le  \frac{2\rho^{N-1}(R-\rho)}{(R^{N-1}-\rho^{N-1})(- \partial_r u_R(R) )}  \varepsilon\le \frac {2\rho}{(N-1)(- \partial_r u_R(R) )}  \varepsilon \le \frac {2R}{(N-1)(- \partial_r u_R(R) )}\varepsilon.
$$
Therefore, it suffices to set $K=\frac {2R_*}{(N-1)(- \partial_r u_{R_*}(R_*) )}$, where $R_*$ is the diameter of $\Omega$ since $R_* \ge R$.
\end{proof}


\section{Proof of Theorem \ref{th:main2}}
\label{sec:proof2}
\begin{proof}[Proof of Theorem \ref{th:main2}]

\textbf{Step 1.} \emph{Transmission conditions.}
 Let $u$ be the solution of \eqref{eq:op2} subject to \eqref{eq:dirichlet}.  We define the functions $u^\pm=u^\pm(x)$ by
 \begin{equation}
 \label{upm}
 u^+=u \ \mbox{ in } \overline{D}\quad \mbox{ and }\quad u^-= u \ \mbox{ in } \overline{B_L(0)}\setminus D.
 \end{equation}
 Then $u^+$ solves the obstacle problem
 \begin{align}\label{eq:op+}
\begin{cases}
\min \{-\Delta u^+, u^+-\psi\}=0 & \text { in } D, \\
u^+=d & \text { on } \partial D,
\end{cases}
\end{align}
and $u^-$ satisfies
\begin{equation}\label{harmonic}
-\Delta u^-=0 \ \text{ in } B_L(0) \setminus \overline{D}, \quad  u^-=d \ \text { on } \partial D, \quad \mbox{ and }  u^-=0 \text { on } \partial B_L(0).
\end{equation}
By the maximum principle, we have
 \begin{equation}
 \label{important simple inequality}
 u^+ > d \quad \mbox{  in }D.
 \end{equation}
In particular, the transmission conditions on the interface $\partial D$ are written as
\begin{equation}
\label{transmission conditions}
u^+ = u^-\ (=d) \quad \mbox{ and } \quad \sigma_+ \partial_\nu u^+ = \sigma_- \partial_\nu u^- \quad \mbox{ on } \partial D,
\end{equation}
where  $\nu$ denotes the outward-pointing unit normal vector to $\partial D$.

\textbf{Step 2.} \emph{Auxiliary problems in symmetric domains.}
  By virtue of \eqref{ass:o2} and  \eqref{supported inside D}, we may choose a $\rho > 0$ satisfying
$$
\overline{\{ x \in \mathbb R^N : \psi(x) > 0\}} \subset B_\rho(0) \subset \overline{B_\rho(0)} \subset D.
$$
Let $u_\rho, u_L$ be the solutions of the following obstacle problems, respectively: 
\begin{align}\label{eq:op4}
\begin{cases}
\min \{-\Delta u_\rho, u_\rho-\psi\}=0 & \text { in } B_\rho(0), \\
u_\rho=0 & \text { on } \partial B_\rho(0);
\end{cases}
\end{align}
\begin{align}\label{eq:op5}
\begin{cases}
\min \{-\Delta u_L, u_L-\psi\}=0 & \text { in } B_L(0), \\
u_L=d & \text { on } \partial B_L(0),
\end{cases}
\end{align}
where $d$ is the positive constant given in \eqref{eq:dirichlet}.
Then,  from the comparison principle, \eqref{important simple inequality}, and Lemma \ref{le:connected1} below, we deduce 
\begin{align*}
\begin{cases}
u_{\rho} \leq u \leq u_{L} & \text { in } B_{\rho}(0), \\
u \leq u_{L} & \text { in } B_L(0).
\end{cases}
\end{align*}
Moreover, as in the proof of Theorem \ref{th:main1}, we observe that
\begin{equation}
\label{shape of auxiliary functions}
\mbox{both } u_{\rho} \mbox{ and }u_{L} \mbox{ are radially symmetric with respect to the origin and non-increasing in $|x|$.}
\end{equation}
For the coincidence sets 
\begin{align*}
\mathcal I &=\{x \in B_L(0): u(x)=\psi(x)\}, \\
\mathcal I_{\rho} &=\left\{x \in B_{\rho}(0): u_{\rho}(x)=\psi(x)\right\}, \\
\mathcal I_{L} &=\left\{x \in B_{L}(0): u_{L}(x)=\psi(x)\right\}
\end{align*}
the following inclusions hold: 
\begin{equation}
\label{three coincidence sets}
\mathcal I_{\rho} \supset  \mathcal I \supset \mathcal  I_{L}.
\end{equation}
Introducing the notation  
\begin{equation}
\label{summit space}
\mathcal I_*:=\{ x \in \mathcal  I_{L} : u_L(x)=\max \psi \} \quad  \mbox{ and } \quad  r_*:= \max\{ |x| : x \in \mathcal I_* \},
\end{equation}
we observe that $\mathcal I_*$ is radially symmetric with respect to the origin and $0\le r_* < \rho$. Remark that 
\begin{equation}\label{shape of the summit of the solution}
u(x) \begin{cases}
 \le u_L(x)< \max \psi &\mbox{ if } r_* < |x| \le L,\\
 = u_L(x)= \max \psi &\mbox{ if } |x| \le r_*.
 \end{cases}
\end{equation}

\textbf{Step 3.} \emph{Penalized problems.} The radially symmetric function $u_\rho$ can be extended as a continuous and non-increasing function in $|x|$ over $\mathbb R^N$   by setting $u_\rho=0$ in $\mathbb R^N\setminus \overline{B_\rho(0)}$.
Then, we also have
$$
u_{\rho} \leq u   \text { in } B_L(0).
$$
Hence, we notice that, even if the obstacle function $\psi$ is replaced by $u_\rho$, the solution $u$ of problem \eqref{eq:op2} does not changes.
For $0<\varepsilon < 1$, let $v^\varepsilon \in C^2(\overline{D})$ be the solution of the penalized problem
\begin{align}\label{eq:penalized problems}
\begin{cases}
-\Delta v^\eps=\beta_\varepsilon(u_\rho-v^\eps) & \text { in } D, \\
v^\eps=d & \text { on } \partial D,
\end{cases}
\end{align}
where $\beta_\varepsilon(t) = \beta(t/\varepsilon)$ with $\beta \in C^2(\mathbb R)$ satisfying
$$
\beta(t)=0 \mbox{ for } t \le 0,\  \beta(t)>0 \mbox{ for } t > 0,\  \beta(t) \mbox{ grows linearly for } t \gg1,\ \beta^\prime\ge 0,  \mbox{ and } \ \beta^{\prime\prime}\ge 0 \mbox{ in }\mathbb R.
$$
By the maximum principle, we have
\begin{equation}\label{approximate solutions in D}
v^\varepsilon > d \quad \mbox{ in } D.
\end{equation}
The family of solutions $\{ v^\varepsilon\}_{\eps >0}$ of the penalized problems \eqref{eq:penalized problems} converges to the solution $u$ in $C^1(\overline{D})$ as $\varepsilon \to 0^+$ (we refer to \cite[Section 2, pp. 209--217]{MR4249421} or the books \cite{MR0679313, MR1786735, MR1094820} for details).

\textbf{Step 4.} \emph{The method of moving planes.} We apply  the method of moving planes due to Serrin (see \cite{MR0544879,MR1463801,MR333220,MR1808026}) to our problem in order to show that $D$ must be a ball centered at the origin. We follow the presentation in \cite{MR4517773} with the aid of the approximate solutions $\{ v^\varepsilon\}_{\eps >0}$.

Let $\gamma$ be a unit vector in $\mathbb R^N,$ $\lambda\in\mathbb R,$ and let $\pi_\lambda$ be the hyperplane $x\cdot\gamma=\lambda,$ where $x\cdot\gamma$ denotes the Euclidean inner product of $x$ and $\gamma$.
For large $\lambda,$ $\pi_\lambda$ is disjoint from $\overline{D}$; as $\lambda$ decreases,
$\pi_\lambda$ intersects $\overline{D}$ and cuts off from $D$ an open cap $D_\lambda= \{ x \in D : x\cdot\gamma > \lambda\}$.
Let us denote by $D^\lambda$ the reflection
of $D_\lambda$ with respect to the plane $\pi_\lambda$. Then, $D^\lambda$ is contained in $D $ at the beginning, and remains in $D$ until one of the following events occurs:
\begin{enumerate}
\item[\textbf{(i)}] $\lambda = 0;$
\item[\textbf{(ii)}] For some $\lambda > 0$, $D^\lambda$ becomes internally tangent to $\partial D$
at some point $p\in\partial D\setminus \pi_\lambda;$
\item[\textbf{(iii)}] For some $\lambda > 0$, $\pi_\lambda$ reaches a position where it is orthogonal to $\partial D$
at some point $q\in\partial D\cap\pi_\lambda$ and the direction $\gamma$ is not tangential to $\partial D$ at every point on $\{ x \in \partial D : x\cdot\gamma > \lambda\}$.
\end{enumerate}

If event \textbf{(i)} occurs for every direction $\gamma$, then $D$ must be a ball centered at the origin and hence $u$ is radially symmetric with respect to the origin.
Therefore, by supposing that either \textbf{(ii)} or \textbf{(iii)} occurs for some direction $\gamma$, we shall derive a contradiction. 

Let us suppose that either \textbf{(ii)} or \textbf{(iii)}  occurs for $\lambda=\lambda_* > 0$. By a rotation of coordinates, we may assume that $\gamma = (1,0,\dots,0)$ and the plane $\pi_{\lambda_*}$ is represented by $x_1=\lambda_*$ with $x=(x_1,\dots,x_N)$. Let us denote by $E, F$ the reflections of $\{ x \in B_L(0)\setminus\overline{D} : x_1 > \lambda_* \}, \{ x \in B_L(0) : x_1 > \lambda_* \}$ with respect to $\pi_{\lambda_*}$, respectively. Let $\Sigma$ be the connected component of $F\cap\{ x \in B_L(0)\setminus\overline{D} : x_1 <\lambda_* \}$ whose boundary contains the points $p$ and $q$ in the respective cases \textbf{(ii)} and \textbf{(iii)}. Since $D^{\lambda_*}\subset D$, we have $\Sigma\subset E$.

Let us denote by $x^{\lambda_*}$ the reflection of a point $x \in \mathbb R^N$ with respect to the $\pi_{\lambda_*}$, i.e. 
$$
x^{\lambda_*} = (2\lambda_*-x_1, x_2, \dots, x_N) \quad  \mbox{ for } x = (x_1,x_2,\dots,x_N).
$$
Let us introduce the functions $w^\pm=w^\pm(x)$ and $w^\varepsilon=w^\varepsilon(x)$: 
\begin{align*}
\begin{cases}
w^+(x) = u^+(x)-u^+(x^{\lambda_*} )\ &\mbox{ for } x \in \overline{D^{\lambda_*}},\\
w^-(x) = u^-(x)-u^-(x^{\lambda_*} )\ &\mbox{ for } x \in \overline{\Sigma},\\
w^\varepsilon(x) =v^\varepsilon(x)-v^\varepsilon(x^{\lambda_*} )\ &\mbox{ for } x \in \overline{D^{\lambda_*}}.
\end{cases}
\end{align*}
Then, since $\lambda_* > 0$, by virtue of Lemma \ref{le:connected1}, we notice that 
\begin{equation}\label{never vanish}
w^-\not\equiv 0\quad \mbox{ in } \Sigma.
\end{equation}
Hence, by observing that 
$$
\Delta w^-= 0\ \mbox{ in }  \Sigma\quad  \mbox{ and } \quad  w^- \ge 0 \ \mbox{ on } \partial \Sigma,
$$
the maximum principle yields 
\begin{equation}
\label{strict positivity of w-}
w^- > 0\quad \mbox{ in }\Sigma. 
\end{equation}
For $x \in D^{\lambda_*}$, since $u_\rho(x) \ge u_\rho(x^{\lambda_*})$,  we compute
\begin{align*}
-\Delta w^\varepsilon(x) &= \beta_\varepsilon(u_\rho(x)-v^\varepsilon(x)) - \beta_\varepsilon(u_\rho(x^{\lambda_*})-v^\varepsilon(x^{\lambda_*})) \\
&\ge \beta_\varepsilon(u_\rho(x)-v^\varepsilon(x)) - \beta_\varepsilon(u_\rho(x)-v^\varepsilon(x^{\lambda_*}))\\
&=-\int_0^1\beta^\prime_\varepsilon(u_\rho(x)-(1-\theta)v^\varepsilon(x)-\theta v^\varepsilon(x^{\lambda_*})) \dd\theta \ w^\varepsilon(x).
\end{align*}
Then, by deducing from \eqref{eq:penalized problems} and \eqref{approximate solutions in D} that $w^\varepsilon \ge 0\mbox{ on }\partial D^{\lambda_*}$, we conclude from the maximum principle that
$w^\varepsilon\ge 0$ in $D^{\lambda_*}$. Thus, letting $\varepsilon \to 0^+$ yields that
\begin{equation}
\label{nonnegaticity of w+}
w^+ \ge 0\quad \mbox{ in }D^{\lambda_*}. 
\end{equation} 
 Moreover, since $\lambda_* > 0$, by virtue of \eqref{shape of the summit of the solution}, we note that  $w^+ \not\equiv 0$ in $D^{\lambda_*}$.

From now on, we distinguish the two cases \textbf{(ii)} and \textbf{(iii)}. 

\noindent\textbf{Case (ii)}. The first equality in \eqref{transmission conditions} yields that $w^+(p)=w^-(p)=0$. Then, it follows from \eqref{strict positivity of w-}, \eqref{nonnegaticity of w+}, and Hopf's boundary point lemma that
\begin{equation}
\label{signs of normal derivatives}
\partial_\nu w^+(p) \le 0 < \partial_\nu w^-(p),
\end{equation}
where we used the fact that $\nu$ is the outward-pointing unit normal vector to $\partial D$ as well as the inward-pointing unit normal vector to $\partial\Sigma$. Thus, from the definition of $w^\pm$, we have that
$$
\partial_\nu u^+(p)\le \partial_\nu(u^+(x^{\lambda_*} ))|_{x=p}\quad \mbox{ and }\quad \partial_\nu u^-(p)> \partial_\nu(u^-(x^{\lambda_*} ))|_{x=p}.
$$
As in  \cite[Eq. (3.7), p. 4]{MR4517773}, the reflection symmetry with respect to the plane $\pi_{\lambda_*}$ yields that
$$
\partial_\nu (u^\pm(x^{\lambda_*} ))|_{x=p} = \partial_\nu u^\pm(p^{\lambda_*})
$$
and, hence, 
\begin{equation}
\label{relationship between normal derivatives}
\partial_\nu u^+(p)\le \partial_\nu u^+(p^{\lambda_*} )\quad \mbox{ and }\quad \partial_\nu u^-(p)> \partial_\nu u^-(p^{\lambda_*} ).
\end{equation}
On the other hand, the second equality in \eqref{transmission conditions} yields that
$$
 \sigma_+ \partial_\nu u^+(p) = \sigma_- \partial_\nu u^-(p) \quad \mbox{ and }\quad   \sigma_+ \partial_\nu u^+(p^{\lambda_*} ) = \sigma_- \partial_\nu u^-(p^{\lambda_*} ),
 $$
 which contradicts \eqref{relationship between normal derivatives}.
 
\noindent\textbf{Case (iii)}.  As in  \cite[p. 5]{MR4517773}, by a rotation of coordinates, we may assume that 
 $$
 \nu(q) =(0,\dots,0,1).
 $$
 Since $\partial D$ is of class $C^2$, there exists a $C^2$ function $\varphi : \mathbb R^{N-1}\to\mathbb R$ such that, in a neighborhood of $q$,  $\partial D$ is represented  as a graph $x_N=\varphi(\hat{x})$ where $\hat{x}=(x_1,\dots,x_{N-1}) \in \mathbb R^{N-1}$, where
 $$
 \nabla \varphi (\hat{q}) = 0 \quad \mbox{ and } \quad \nu =\frac 1{\sqrt{1+|\nabla\varphi|^2}}(-\nabla\varphi,1).
 $$
 Then, by proceeding as in  \cite[pp. 5--6]{MR4517773} with the aid of \eqref{transmission conditions}, the harmonicity of $w^-$, \eqref{strict positivity of w-}, and \eqref{nonnegaticity of w+}, we are able to get a contradiction between the transmission conditions \eqref{transmission conditions} and  Serrin's corner point lemma (see \cite[Lemma S, p. 214]{MR0544879} or \cite[Serrin's Corner Lemma, p. 393]{MR1463801}).
\end{proof}

\begin{lemma}\label{le:connected1}
The set $B_L(0) \setminus \overline{D}$ is connected and $0 < u^- < d \mbox{ in } B_L(0) \setminus \overline{D}$.
\end{lemma}

\begin{proof}
Since $d > 0$ in \eqref{eq:dirichlet} and $u=0$ on  $\partial B_L(0)$, there is only one component of  $B_L(0) \setminus \overline{D}$ whose closure intersects $\partial B_L(0)$. Let us suppose, for the sake of finding a contradiction, that another component $\omega$ exists. Then, $\partial\omega\subset\partial D$ and $u = d$ on $\partial\omega$. Since the coincidence sets are contained in $D$, we have, in particular, $\Delta u = 0$ in $B_L(0) \setminus \overline{D}$.  Hence, by the maximum principle, $u \equiv d$ in $\omega$.  Since $\partial D$ consists of regular hypersurfaces, $\partial\omega$ is a part of them. Therefore, \eqref{transmission conditions} implies that  $\partial_\nu u^+ = 0$ on $\partial\omega$, which contradicts \eqref{important simple inequality} and Hopf's boundary point lemma. 

Having shown that $B_L(0) \setminus \overline{D}$ is connected, the latter half of the claim follows from the maximum principle.
\end{proof}


\section{Proof of Theorem \ref{th:main3}}
\label{sec:proof3}

\begin{proof}[Proof of Theorem \ref{th:main3}] By replacing the boundary condition of \eqref{eq:op2} by that of \eqref{eq:op3} at infinity, we can prove  Theorem \ref{th:main3} along the same steps as Theorem \ref{th:main2}. Indeed,  in {Step 1}, $u^-$ in \eqref{upm} and  the boundary condition on $\partial B_L(0)$ in  \eqref{harmonic} may be replaced by 
$$
 u^-=u \ \mbox{ in }\mathbb R^N\setminus D \quad \mbox{ and  }\quad u^- \to 0 \ \mbox{ as } |x| \to \infty,
 $$
 respectively. Also, Lemma  \ref{le:connected1} may be replaced by the following result. 
 
\begin{lemma}\label{le:connected2}
The set $\mathbb R^N \setminus \overline{D}$ is connected and $0 < u^- < d \mbox{ in } \mathbb R^N \setminus \overline{D}$.
\end{lemma}

In {Step 2}, $u_L$ solving \eqref{eq:op5} may be replaced by $u_\infty$ solving
\begin{align}\label{eq:op6}
\begin{cases}
\min \{-\Delta u_\infty, u_\infty-\psi\}=0 & \text { in } \mathbb R^N, \\
u_\infty \to d & \text { as } |x| \to \infty.
\end{cases}
\end{align}
The other some minor replacements corresponding to the above changes will eventually give a contradiction in each case.
\end{proof}

\vspace{5mm}
\section*{Acknowledgments}

N. De Nitti is a member of the Gruppo Nazionale per l'Analisi Matematica, la Probabilit\'a le loro Applicazioni (GNAMPA) of the Istituto Nazionale di Alta Matematica (INdAM). He has been partially supported by the Alexander von Humboldt Foundation and by the TRR-154 project of the Deutsche Forschungsgemeinschaft (DFG). Moreover, he acklowledges the kind hospitality of Tohoku Univeristy, where part of this work was carried out.

S. Sakaguchi has been supported by the Grants-in-Aid
for Scientific Research (B) and (C) (\# 18H01126 and \# 22K03381)  of
Japan Society for the Promotion of Science.

The authors would like to thank Professor Kazuhiro Ishige for suggesting that a stability result, such as Theorem \ref{th:main1-2}, follows directly from the proof of Theorem \ref{th:main1}.

\vspace{5mm}

\bibliographystyle{abbrv}
\bibliography{StarOP-ref.bib}

\vspace{3mm}
\vfill

\end{document}